\documentclass[11pt]{amsart}
\usepackage{setspace}
\usepackage[top=2.0cm,bottom=2.0cm,left=2.5cm,right=2.5cm]{geometry}

\usepackage[utf8]{inputenc} 
\usepackage{booktabs}
\usepackage{graphicx} 
\usepackage{hyperref}
\usepackage{amsmath}
\usepackage{amsfonts}
\usepackage{amssymb}
\usepackage{amsthm}
\usepackage{enumitem}
\usepackage{color}	
\usepackage{todonotes}

\definecolor{melonBase}{RGB}{255,145,116}
\colorlet{pastelMelon}{melonBase!20!white}

\colorlet{perfectBlue}{blue!60!white}

\usepackage[labelformat=simple]{subfig}
\newcommand{\B}{\mathbb B}
\newcommand{\N}{\mathbb N}
\newcommand{\Z}{\mathbb Z}
\newcommand{\R}{\mathbb R}
\newcommand{\E}{\mathbb E}
\DeclareMathOperator{\ch}{ch}
\DeclareMathOperator{\ess}{ess}
\DeclareMathOperator{\supp}{supp}
\DeclareMathOperator{\var}{var}
\DeclareMathOperator{\diam}{diam}
\DeclareMathOperator{\divr}{div}
\DeclareMathOperator{\lip}{Lip}
\DeclareMathOperator{\interior}{int}
\newtheorem{lem}{Lemma}

\newtheorem{thm}{Theorem}
\newtheorem{athm}{Theorem}
\renewcommand{\theathm}{\Alph{athm}}
\newtheorem{prop}{Proposition}
\newtheorem{cor}{Corollary}
\newtheorem{claim}{Claim}

\theoremstyle{definition}

\newtheorem{defn}{Definition}

\theoremstyle{remark}
\newtheorem{rmk}{Remark}

\usepackage{amsmath,amssymb,amsthm}

\title{Optimal Play in Hex on Finite and Infinite Boards}

\author[S. Casco]{Samuel Casco}
\address{Instituto de Matemática, Universidade Federal do Rio de Janeiro - UFRJ, Cidade Universitária - Ilha do Fundão, Rio de Janeiro  21945-909. Brazil}
\email{welington@unifei.edu.br}
\thanks{S. Casco was partially supported by CAPES-Finance Code 001 and by FAPERJ Grant number 173057/2023-3}
\author[M. J. Pacifico]{Maria José Pacifico}
\address{Instituto de Matemática, Universidade Federal do Rio de Janeiro - UFRJ, Cidade Universitária - Ilha do Fundão, Rio de Janeiro  21945-909. Brazil}
\email{pacifico@im.ufrj.br}
\thanks{M. J. Pacifico was partially supported by CAPES-Finance Code 001, CNPq Projeto Universal No. 404943/2023-3, CNPq-Brazil grant 307776/2019-0 and by
	Foundation for Research Support of the State of Rio de Janeiro (FAPERJ) grant CNE
	E-26/200.913/2022(268181).}
\author[D. Sanhueza]{Diego Sanhueza}
\address{Instituto de Matemática, Universidade Federal do Rio de Janeiro - UFRJ, Cidade Universitária - Ilha do Fundão, Rio de Janeiro  21945-909. Brazil}
\email{xuanz@id.uff.br}
\thanks{D. Sanhueza was partially supported by CAPES-Finance Code 001, FAPERJ Grant-E-26/200.038/2025 and CAPES-PrInt 88887.840543/2023-00.}

\begin{document}

\maketitle

\begin{abstract}

The game of Hex is one of the most celebrated connection games in combinatorics.
Although it is known that the first player always has a winning strategy, very little is understood about the complexity of optimal play.

Following Campbell's introduction of the parameters
$\lambda(n)$ and $\delta(n)$, measuring respectively
the length of the shortest guaranteed winning path and the minimum number of stones required by the first player to force a win,
we establish new structural properties of these quantities.

Our main result determines the exact value
\[
\lambda(5)=7,
\]
thereby confirming one of Campbell's conjectures.

We also show that the corresponding problem on the infinite strip
$5\times\infty$ has a markedly different behavior. In this setting we prove that
\[
\lambda(5\times\infty)=5,
\]
which is strictly smaller than $\lambda(5)$.

To the best of our knowledge, this paper provides the first solution to one of Campbell's 2004 conjectures and answers a question posed by Stromquist (2006) in the affirmative. Furthermore, it uncovers a fundamental distinction between finite and infinite boards, offering new insights into the nature of optimal winning strategies in Hex.
\end{abstract}
\tableofcontents
\section{Introduction} 
Hex is one of the most celebrated connection games in combinatorial
game theory. It was invented independently, and almost simultaneously,
by two mathematicians with very different backgrounds: Piet Hein, a
Danish poet and scientist, introduced the game in $1942$ under the
name \emph{Polygon}; a few years later, unaware of Hein's work, John
Nash rediscovered it as a graduate student at Princeton, where it
became known among his colleagues as \emph{Nash} or
\emph{John}. This dual origin --- part poetry, part game theory ---
already hints at the peculiar charm of Hex: a game whose rules are
almost childishly simple, yet whose strategic depth has occupied
mathematicians for over eight decades.

The game is played on an $n\times n$ rhombus tiled by regular
hexagons. Two players, whom we call Blue and Red, alternately occupy
previously empty cells, each attempting to connect a designated pair
of opposite sides of the board with the other's own; a player wins
as soon as a connected chain of their own stones joins their two
sides.

A remarkable feature of Hex, and one of the reasons it has attracted sustained mathematical attention, is that draws are impossible: every completed Hex board contains exactly one winning connection. This fundamental result, known as the Hex Theorem, was proved in \cite{Ga79}, which also established its striking equivalence with the Brouwer Fixed Point Theorem.
 As an immediate consequence, Nash's celebrated strategy-stealing argument
shows that the first player always possesses a winning strategy on
every finite board. This theorem is as elegant as it is
unsatisfying: it guarantees a winning strategy exists while offering
not the slightest hint of how to actually play it. 
Explicit winning strategies are known only for relatively small boards, and the computational complexity of finding optimal play grows rapidly with the board size. This is not surprising, since determining the winner in Hex was shown by Reisch \cite{Rei81} to be PSPACE-complete (see also Even and Tarjan \cite{ET76}). Consequently, most previous research has focused either on determining winning moves for small boards, by hand or with computer assistance (see, for instance, \cite{Mi06, No04, No05, Ya01}), or on developing heuristics for computer play (e.g. \cite{HAH09, HBJKPR05}). Comparatively little attention has been devoted to establishing structural, mathematically rigorous properties of optimal strategies.

In contrast, Campbell \cite{Ca04} proposed a different point of
view.  Rather than attempting to describe a winning strategy explicitly, he introduced two numerical invariants that measure the intrinsic
\emph{complexity} that any winning strategy must carry. The first,
$\lambda(n)$, is the minimum length of a winning path that the first
player can guarantee, regardless of the opponent's play. The second,
$\delta(n)$, measures a subtler quantity: the minimum number of
stones that are strictly indispensable to a forced win, over all
 strategies and all of the opponent's replies. Together these
parameters shift the question from ``how does one win at Hex'' to
``how \emph{simple} can a forced win be made to be'' --- a question
that, perhaps surprisingly, turns out to be genuinely difficult even
for boards as small as $5\times 5$.

Campbell computed the exact values
\[
\lambda(2)=2,\qquad \lambda(3)=3,\qquad \lambda(4)=5,
\]
and
\[
\delta(2)=2,\qquad \delta(3)=3,\qquad \delta(4)=5.
\]
For the $5\times5$ board, he established the bounds
\[
6\le \delta(5)\le 7,
\]
and conjectured that the upper bound is attained. This conjecture has remained open for almost two decades.

The main goal of the present paper is to settle this conjecture.

\begin{athm}\label{t-main}
\[
\lambda(5)=7.
\]
\end{athm}

Our proof is entirely elementary and does not rely on exhaustive computer search. Instead, we identify a small collection of local tactical patterns—called BR-blocks, a Hex-theoretic analogue of what are known elsewhere as {\em carriers of virtual connections} and use them to organize a concise, symmetry-reduced analysis of Blue's possible opening moves. In every case, the corresponding BR-blocks yield barriers that force any winning path to contain at least \textcolor{blue}{seven stones.}

The same patterns yield, as a
by-product, a unified strategy achieving the matching upper bound
$\lambda(n)\le 2n-3$ for every $n\in\{3,4,5\}$, suggesting that the
equality $\lambda(n)=\delta(n)=2n-3$ may persist well beyond the
cases settled here.

We also answer a question posed by  Stromquist \cite{St07}, showing that the corresponding problem on the infinite strip $5\times\infty$ exhibits a fundamentally different behavior.

\begin{athm}\label{t-main2}
\[
\lambda(5\times\infty)=5.
\]
\end{athm}

 Since $\lambda(n\times\infty)\geq n$, the proof of Theorem B falls on defining an explicit strategy with length $5$.
To the interested reader on Hex, we recommend \cite{HR06, HBJKPR05} and references therein. 
\vspace{0.2cm}

The remainder of the paper is organized as follows.
Section~\ref{s-preliminaries} recalls the basic definitions of
$\lambda(n)$ and $\delta(n)$ and records the elementary monotonicity
and embedding facts we rely on later. Section~\ref{s-BR-blocks} introduces
$BR$-blocks and the auxiliary lemmas on the $3\times3$ and
$4\times4$ boards that make the main case analysis tractable.
Section~\ref{s-proof-teoA} contains the proof of Theorem~\ref{t-main},
organized by the opening move of the first player. 
Section~\ref{s-infty} proves Theorem \ref{t-main2}. 
We close with
some remarks on the values $\delta(5)$, the general upper bound
$\lambda(n)\le 2n-3$, and the open problems this leaves behind.

\section{Preliminaries}\label{s-preliminaries} 
We begin by describing the standard $n\times n$ Hex board and introducing the notation and terminology used throughout the paper. Two opposite sides of the board are colored blue, while the other two are colored red.

Two players, called \emph{Blue} and \emph{Red}, alternately occupy previously empty hexagonal cells, with Blue moving first. Blue wins by constructing a connected chain of blue stones joining the two blue sides of the board, while Red attempts to connect the remaining pair of opposite sides. We assume familiarity with the basic properties of Hex and refer the reader to \cite{Ca04,Ga79} for further background. In particular, we shall repeatedly use the fact that draws are impossible and, consequently, that the first player always possesses a winning strategy. 
 
 \subsection{Winning paths} 
  A \emph{winning path} is a connected chain of stones joining the two sides assigned to the winning player. More precisely, a winning path is a sequence of cells \[\gamma= c_1,c_2,\ldots,c_k, \] such that \begin{itemize} \item $c_1$ belongs to one designated side of the board; \item $c_k$ belongs to the opposite designated side; \item consecutive cells are adjacent; \item every cell is occupied by the winning player. \end{itemize} The \emph{length} of a winning path $\gamma$ is the smallest number of essential stones it contains and it is denoted by $|\gamma|$. Different winning plays generated by the same strategy may produce winning paths of different lengths.  We denote by $\lambda(S)$ the greatest length of a winning path generated by the strategy $\mathcal{S}$. The following parameter measures the shortest winning connection that the first player can always guarantee. 
  
  \begin{defn} The quantity $\lambda(n)$ is the smallest integer $k$ such that the first player has a strategy guaranteeing that every resulting winning path in the $n\times n$ Hex board contains at most $k$ stones. 
  \end{defn} 
 Since every path joining opposite sides of an $n\times n$ board must meet at least one cell in each row, we always have that $ \lambda(n)\ge n$.  Actually, it is known from Campbell's work \cite[Theorem 4]{Ca04} that $\lambda(n)$ is strictly greater that $n$ for any $n\geq4$. Stromquist \cite{St07} obtains, later, that $\lambda(n)\geq n+\lfloor \frac{n-2}{4}\rfloor $ implying that
$$
\lim_{n \to \infty} (\lambda(n)-n)= \infty.
$$       

  \subsection{Winning sets} 
  Not every stone played during a winning game is necessarily essential. Some stones merely support the strategy, whereas others are indispensable for the final winning connection. 
  

  \begin{defn} Let $\mathcal S$ be a winning strategy for Blue and consider a complete play against an arbitrary strategy of Red. A \emph{winning set} is the collection of Blue's stones used in the game. We denote by $\delta(\mathcal{S})$ the maximum cardinality of a winning set generated by $\mathcal{S}$.
  \end{defn}

  Campbell introduced the following parameter.
   \begin{defn} The quantity $\delta(n)$ is the minimum, over all winning strategies for the first player, of the maximum cardinality of a winning set produced against any possible response of the second player. In other words, $\delta(n)=\min\{\delta(\mathcal{S}): \mathcal{S}\mbox{ is a winning strategy for the first player}\}.$ \end{defn} 
   Thus, while $\lambda(n)$ measures the length of the shortest guaranteed winning path, $\delta(n)$ measures the minimum number of essential stones that every winning strategy must eventually produce. The two parameters satisfy an immediate relation. 
  
  \begin{prop}\label{prop:lambda-delta} For every $n\ge2$, \[ \lambda(n)\le\delta(n). \] 
  \end{prop} 
  \begin{proof} Every winning set contains a winning path. Consequently, the length of that path is at most the cardinality of the winning set. Taking the minimum over all winning strategies yields the desired inequality. 
  \end{proof} 
  \subsection{Known results} 
   Campbell \cite{Ca04} computed the exact values \[ \lambda(2)=2,\qquad \lambda(3)=3,\qquad \lambda(4)=5, \] and \[ \delta(2)=2,\qquad \delta(3)=3,\qquad \delta(4)=5. \] For the $5\times5$ board Campbell proved \[ 6\le \lambda(5)\le \delta(5)\le7, \] which naturally led to the conjecture \[ \lambda(5)=\delta(5)=7. \] The purpose of this paper is to establish this equality. 
  \section{BR-blocks}\label{s-BR-blocks}

Visually, we can see that two distinct stones, say $c_1$ and $c_2$, have distance $d(c_1,c_2)=2$ if there is a stone $c$ so that $c_1cc_2$ is an {\it essential} $3$-path. If $d(c_1,c_2)=2$ then there are at most two different stones $c$ and $c'$ guaranteeing a connection between $c_1$ and $c_2$. 
These two stones are necessarily adjacent. In the Hex literature, such a
configuration is usually regarded as the carrier of a virtual
connection between $c_1$ and $c_2$. Throughout this paper, however, we
shall call it a \emph{$BR$-block}. The reason for this change in
terminology is that our use of these configurations goes beyond virtual
connections. Besides serving as carriers, $BR$-blocks will also be used
to construct barriers that constrain the geometry of Blue's winning
paths, allowing us to derive lower bounds on the length of any winning
path.

The importance of considering these $BR$-blocks lies in the following pairing property called a {\it virtual connection}. Suppose that Blue already occupies $c_1$ and $c_2$. 
If Red plays one of the two stones in the $BR$-block, then Blue can immediately respond by occupying the other one, thereby guaranteeing a connection between
$c_1$ and $c_2$. The analogous statement holds with the roles of Blue and Red interchanged.\\

 We can define {\it virtual connection} in a more general way. Two cells $c_1$ and $c_2$ form a virtual connection for Blue if there is a sequence of couple of unoccupied cells such that taking any one of them in each couple it is possible to obtain a Blue path connecting $c_1$ and $c_2$. The first part describes the most simple virtual connection called a {\it bridge} and it is showed at Figure \ref{f-BR-blocks}. Note that any virtual connection is a composition of bridges.

\begin{figure}[h]
\begin{center}
\includegraphics[width=14cm,height=3cm,angle=0]{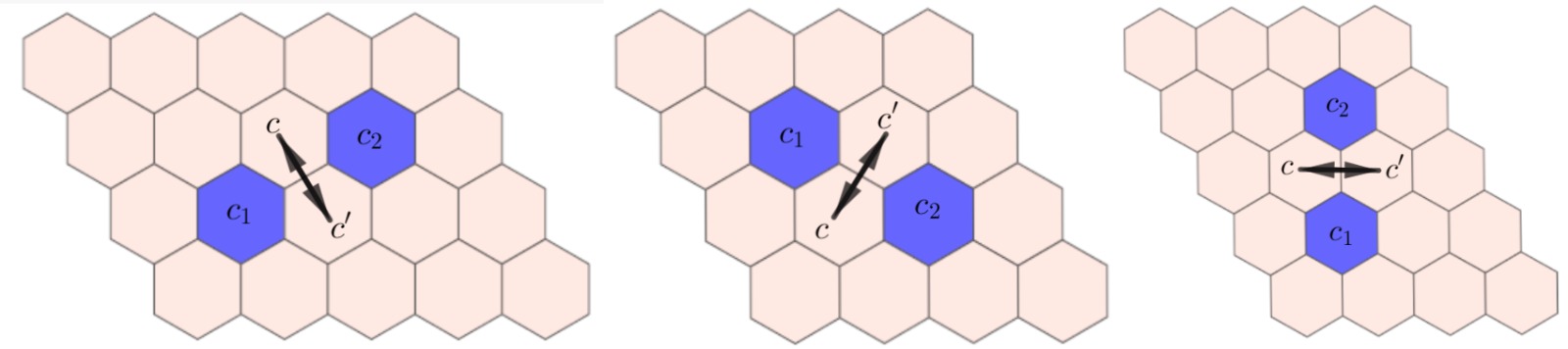}
\caption{Virtual connection. Three different $BR$-blocks connecting $c_1$ and $c_2$.}\label{f-BR-blocks}
\end{center}
\end{figure}

If $c=(i,j)$ with $i=4$ and $1\leq j\leq n-2$, we define the set $X^+_c$ being the collection of the cells $(3,j), (3,j+1), (2,j+1), (1,j+1)$, and $(1,j+2)$. Whether a cell $c$ is taken in this position and the set $X_c^+$ is free, we can take the cell $(2,j+1)$ and assure a path of length $4$ from $c$ to the top, by creating 2 $BR$-blocks. In a similar manner, we can define the set $X^{-}_c$ to connect $c$ with the bottom. We put $X^1_c$ if we are in position to create a path of length $4$ from $c$ to the right side and $X^{-1}_c$ for the left side, see Figure \ref{f-X-Zone}.  In this case, we say that there is a {\it winning virtual connection at $c$ through $(2,j+1)$}.

\begin{figure}[h]
\begin{center}
\includegraphics[width=11.0cm,height=3.0cm,angle=0]{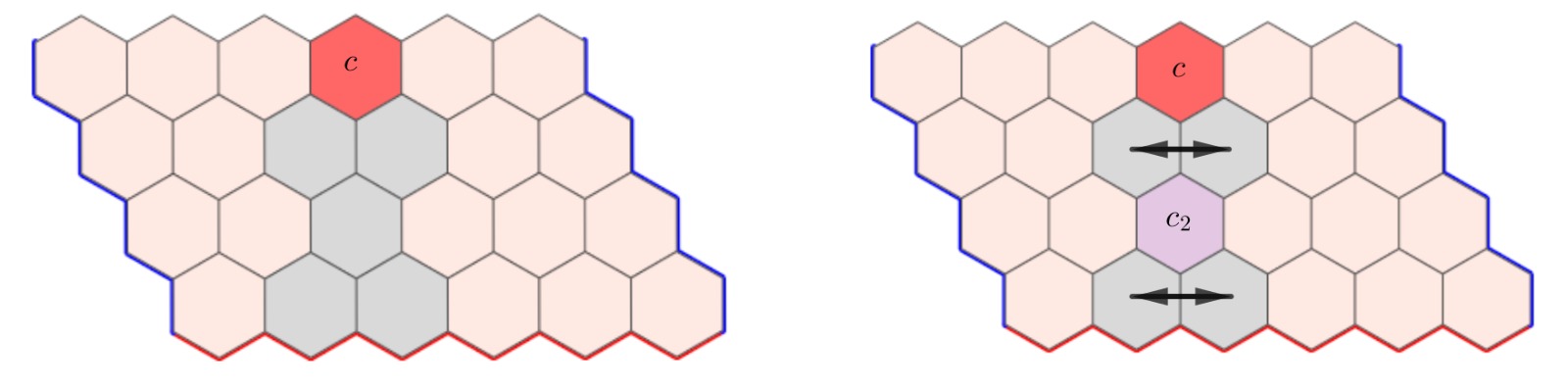}
\caption{The region $X^{-}_c$ generated from the cell $c$ (left). A winning virtual connection at $c$ through $c_2$.}\label{f-X-Zone}
\end{center}
\end{figure}
   
\begin{lem}[{\bf $3$-Center strategy}]\label{l-center} In a $3\times 3$ board, there is a unique optimal winning strategy starting in the center.
\end{lem}
By unique in this lemma we undestand that $\delta(3)=\lambda(3)=3$ and that a winning path has the form 
$$
(2,2)\rightarrow (BR-\mbox{block})\rightarrow(BR-\mbox{block}).
$$
{\textit Proof.} It is  very natural because once Blue choice the center cell he creates two $BR$-blocks each connecting the center with the assign sides. Hence if Red chooses to play in the column $1$ or $2$, then Blue takes an unoccupied place of the left $BR$-block $(2,1)-(3,1)$. If Red plays in the third column, Blue plays in a unoccupied stone of the right $BR$-block $(1,3)-(2,3)$. Note that these movements are independent, meaning to play in columns $1$ or $2$ does not interfers in the game of the column $3$. So, for any choices of Red, Blue creates a winning path of the form
$$
(2,2)\rightarrow (BR-\mbox{block})\rightarrow (BR-\mbox{block}).
$$
This is the unique optimal strategy because any winning path, when the game iniciates in the center must be of the above form, since any winning ending out of the $BR$-blocks ending either the cell $(1,1)$ or $(3,3)$, or both, having a distance greater than $2$ to the center implyinig that the path has a lengh at least $4$. \hspace{14cm} \qquad$\blacksquare$

\begin{lem}  In a $4\times4$ Hex board, the first player always has a path connecting the cell $(3,2)$ (resp. $(2,3)$) with the right (resp. left) side. Furthermore, this path is contained in a winning path.
\end{lem}
{\textit Proof.} In these conditions, it is possible to define the set $X_c^1$ (resp. $X_c^{-1}$) where $c=(3,2)$ (resp. $c=(2,3)$), then the is a path connecting $c$ and the right (left) side of the board, as already explained before. In the first column is defined the $BR$-block $[(3,1)(4,1)]$. \hspace{4,5cm} \qquad$\blacksquare$\\

The following is a 'vertical version' of the above lemma.

\begin{lem}\label{X-zone} In a $4\times4$ Hex board, for any play starting at either $(1,1)$, $(2,1)$ or $(3,1)$ the second player can construct a path connecting the cell $(2,3)$ with the bottom. 
\end{lem}

\begin{lem}\label{Caixa3x3} 
Suppose a $3\times 3$ board in which the first player has possibly occupied all (or some of) the cells in the set $\{(1,1), (1,2), (2,1)\}$, but the Red player occupied the cell $(1,3)$. Then the unique next movement for the first player to obtain a winning path is choosing the cell $(3,2)$. That is, in any other choice the Red player winning the game. 
\end{lem}
{\textit Proof.} If Blue chooses $(2,2)$ or $(3,1)$, then red take $(2,3)$ and the $BR$-block [(3,2)(3,3)]. If Blue takes $(2,3)$ then red chooses $(2,2)$ and the $BR$-block [(3,1)(3,2)]. If Blue moves to $(3,3)$, then red moves to $(3,2)$ and use the $BR$-block $[(2,2)(2,3)]$. In any of the preceding cases, red connects the top with the bottom destroying any connected path from the left to the right. The remaining choice $(3,2)$ for Blue prevents any red winning connection because whether red plays in the second row, blue take the corresponding stone but in the third row, and conversely. Since Red has no winning path, by the Hex Theorem, Blue win the game. \hspace{7,5cm} \qquad$\blacksquare$

\begin{lem}[{\bf $5$-Center Strategy}] Any Strategy in the $5\times 5$ board iniciating in the center define a winning path of length at least $7$.
\end{lem}
{\textit Proof.} We first define a natural strategy $S$ when $b_1=(3,3)$. We divide the remaining possible movements for Red into two sets, $X_1=\{(i,j):1\leq i\leq5, 1\leq j\leq 2\}\cup \{(3,1), (3,2)\}$ and $X_2=\{(i,j):1\leq i\leq 5, 4\leq j\leq5\}\cup\{(4,3)(5,3)\}$. The game is independent when playing in either the set $X_1$ or $X_2$. By the rotational symmetry of the game, we only explain the strategy when the game run in $X_1$. If $r_1=(4,1)$, then $b_2=(2,2)$ and are formed the $BR$-blocks $[(3,2)(2,3)]$. Note that in this case, the path connecting the center to the left right has $4$ stones. If $r_1\in\{(3,1),(3,2)\}$, then blue take the cell $(4,2)$ and use the $BR$-block [(4,1)(5,1)] creating a path from the center to the left with only $3$ stones. If $r_1\not\in\{(3,1)(3,2)(4,1)\}$ then $b_2=(3,2)$ and [(3,1)(4,1)] is a $BR$-block. For this strategy we see that $\lambda(S)=\delta(S)=7$ and is only formed when the the first movement for red in $X_1$ (or $X_2$) is $(4,1)$ (or $(2,5)$), because the only case to obtain a path of lenght $4$ from the center to the left (right) side.\\

Now suppose that $S'$ is a strategy openning in the center forming a winning path with at most $6$ stones. If we denote by $\gamma$ this winning path, then we have that $\# X_i\cap \gamma=2$ for some $i=1,2$. Without lost of generality, we assume that $\# X_1\cap \gamma=2$, then $b_2,b_3\in\{(3,1),(3,2)\}$ or $b_2,b_3\in\{(5,1),(4,2)\}$. But in both cases red can block the path by chossing $r_2=b_3$, then $\#\gamma\cap X_1,\#\gamma\cap X_2\geq3$ and hence $|\gamma|\geq 7$. The proof is now complete.\hspace{7,5cm} \qquad$\blacksquare$\\

 \section{Proof of Theorem \ref{t-main}}\label{s-proof-teoA}
In this section we prove the Theorem \ref{t-main}. Note first that if $\gamma$ is a winning path in a $n\times n$ board, then its lenght is
$$
|\gamma|=\sum_{j=1}^n \#(\gamma\cap C_j)
$$
where $C_j$ is the $jth$ column of the board. Then, to prove Theorem \ref{t-main} we need to show that $\gamma$ intersects some column in at least three stones or $\gamma$ intersects two (or more) columns at least two times.

We start by considering  the collection $\mathfrak{S}(5)$ of all the strategies in the $5\times5$ Hex board. We need to prove that for any winning strategy $S\in\mathfrak{S}(5)$ the winning path formed by this strategy has a length of at least $7$. 
In other words, we shall prove that for any strategy $S$ contain a game with a winning path of length greater or equal to $7$. Since the first player has always a winning stragegy, the proof consists of defining strategies for the second player in order to prevent any winning path with short length. We shall examine each strategy depending on its opening move. We then divide the collection $\mathfrak{S}(5)$ into four categories depending on this first movement: Denote by $\mathfrak{S}_C(5)$ the strategies starting at the center of the board, that is, starting in $(3,3)$. Put $\mathfrak{S}^+(5)$ the collection of all the strategies with opening move in the cell $C(i,j)$, where $1\leq i,j\leq3$ but $(i,j)\neq (3,3)$. We also write $\mathfrak{S}^-(5)$ the collection of the strategies iniciating in $C(i,j)$ with $4\leq i\leq 5$ and $1\leq j\leq 2$. Finally, $\mathfrak{S}_0(5)$ stand for the remaining possibilities. See Figure \ref{f-ThmA-Categorias}.

\begin{figure}[h]
\begin{center}
\includegraphics[width=6.0cm,height=4.0cm,angle=0]{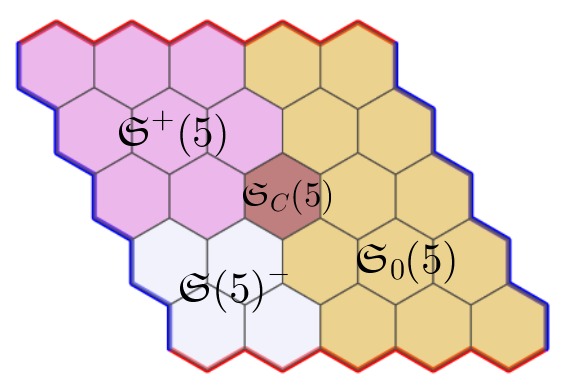}
\caption{Categories are divided depending on the opening move of a strategy.}
\label{f-ThmA-Categorias}
\end{center}
\end{figure}

Now we construct a game such that the winning path obtained after using a given strategy $S$ has length at least $7$. In what follows, we understand Blue and Red as first and second player, respectively. 

{\bf Category $\mathfrak{S}_C(5)$.} Any strategy in this category has a length of at least $7$ by the $5$-Center Strategy Lemma.\\

{\bf Category $\mathfrak{S}^+(5)$.} Denote by $b_1$ the opening move of a fixed strategy $S\in \mathfrak{S}^+(5)$, adopted by the first player (Blue). In this case, Red can prevent any strategy passing completely in the sub-board $1\leq i\leq 3$ by considering the $c=r_1(2,4)$ cell. At the same time, this choice define the region $X_c^{-}$ in the inferior part of the board. We continuing our analysis depending on the second choice $b_2$ for Blue.\\

{\bf I.} If $b_2$ is outside $X_c^{-}$, then Red win the game playing within the set $X_c^{-}$ assuring a winning path through this set by Lemma \ref{X-zone}.
\vspace{0.2cm}

{\bf II.} If $b_2=(3,4)$ then Red can take $r_2=(3,3)$ and then, using Lemma \ref{Caixa3x3} in the inferior $3\times 3$-sub-board, the following sucession is mandatory for Blue considering strategical choices for Red:
$$
b_1cb_2r_2\rightarrow b_3(5,2)\rightarrow r_3(4,3)\rightarrow b_4(5,3)\rightarrow r_4(5,1)\rightarrow b_5(4,2)\rightarrow r_5(4,1)\rightarrow...
$$    
The above selection says that any winning path for $S$ contain at least $3$ stones in the second column. This shows that $\lambda(S)\geq7$. 
\vspace{0.2cm}

{\bf III.} If $b_2\in X_c^{-}\setminus\{(3,4),(5,3)\}$, then Red takes $r_2=(3,5)$ and the following sequence is mandatory for Blue:
$$
b_1cb_2r_2\rightarrow b_3(5,4)\rightarrow r_3(4,5)\rightarrow b_4(5,5)\rightarrow r_4(5,3)\rightarrow b_5(4,4)\rightarrow r_5(5,1)\rightarrow...
$$  
The above sequence only allows winning path for blue starting in $(i,1)$, $i=1,2,3$ and finishing in $(5,5)$. This implies that any (essential) winning path for Blue contain intersects two columns with at least two cells. Hence $\lambda(S)\geq7$.
\vspace{0.2cm}

{\bf IV.} If $b_2=(5,3)$, Red takes $r_2=(3,5)$ and the idea follows the same above argument with the only difference that now Blue has two different movement (namely $(4,4)$ and $(5,4)$). So any possible winning path for Blue starts in $(i,1)$, $i=1,2,3$ and finishes in $(5,5)$ showing that $\lambda(S)\geq7$.
\vspace{0.2cm}

{\bf Category $\mathfrak{S}^-(5)$.}  To illustrate the idea of the proof, we divide the board in three regions. The first ($R_1$) formed by the cells $(i,j)$ with $1\leq i\leq 5$ and $1\leq j\leq 2$. The second region ($R_2$) consisting of cells $(i,j)$ with $1\leq i \leq 2$ and $3\leq j\leq 5$. And the final region ($R_3$) is composed of cells $(i,j)$ with $3\leq i,j\leq5$. See Figure \ref{f-regiones}.

\begin{figure}[h]
\begin{center}
\includegraphics[width=9.0cm,height=3.0cm,angle=0]{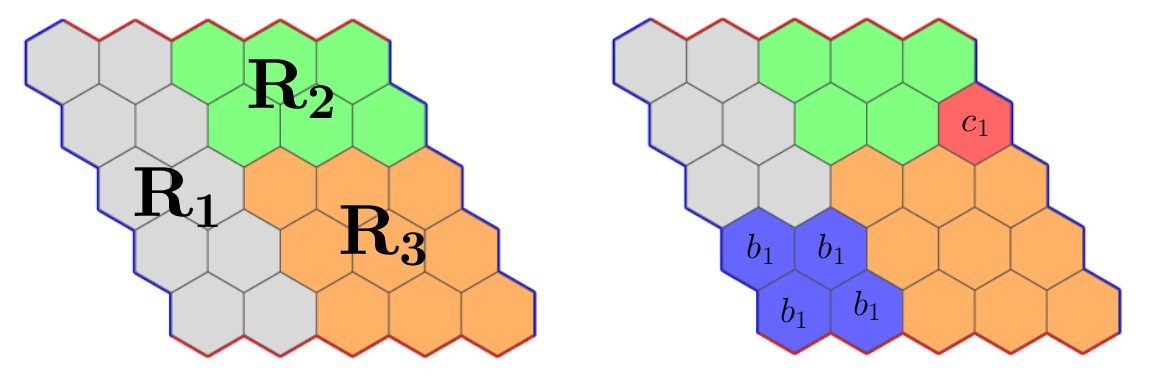}
\caption{The regions $R_1$, $R_2$ and $R_3$.}
\label{f-regiones}
\end{center}
\end{figure}
If $b_1$ denotes the first movement of a fixed strategy $S\in \mathfrak{S}^-(5)$, adopted by the first player (Blue), Red takes the place $c=r_1(2,5)$. The strategy for a good game for Red is pay attention primarily on the third region where is well define the set $X_c^{-}$. Now our analysis focuses on the second and third movement of the first player.
\vspace{0.2cm}

{\bf I.} We start analising the first region. We assume that $b_2\in R_1$. In this case, Red may take $r_2=(4,4)\in X_c^{-}$, and win the game creating the following winning path (after particular choices of Red and disregarding the order as the $BR$-blocks were completed). For instance, \\
$$
b_1r_1b_2r_2\rightarrow b_3(1,5)\rightarrow r_3(2,3)\rightarrow b_4(2,4)\rightarrow r_4(3,3)\rightarrow BR-\mbox{blocks}
$$
representing a maximal winning path for Red whenever $b_2\in R_1$.
 \vspace{0.2cm}
      
{\bf II.} $b_2\in R_2$. If $b_2\neq (1,5)$, then $r_2(1,5)$ and the game runs in the region $R_3$.\\
If $b=(1,5)$, then Red may take $r_2=(2,4)$. 
\vspace{0.2cm}

Now, we anylise what happens if the game take place only in the rows $1$ and $2$. We shall use in the other cases, so we give a name for this situation.
\vspace{0.2cm}

{\bf Strategy $A_1$.} (Assuming $b_1\neq(4,1)$). Suppose that $b_2=(1,5)$, $r_2=(2,4)$ and the cells $(1,1)$ and $(1,2)$ form a $BR$-block. Assume that the game runs only in the rows $1$ and $2$. Then consider the sequence
$$
b_1r_1b_2r_2\rightarrow b_3(1,4)\rightarrow r_3(2,3)\rightarrow b_4(1,3)\rightarrow r_4(2,1)\rightarrow b_5(2,2)\rightarrow r_5(3,1)\rightarrow b_6(3,2)\rightarrow r_6(4,1)\rightarrow BR-\mbox{block}.
$$

\begin{figure}[h]
\begin{center}
\includegraphics[width=9.0cm,height=3.0cm,angle=0]{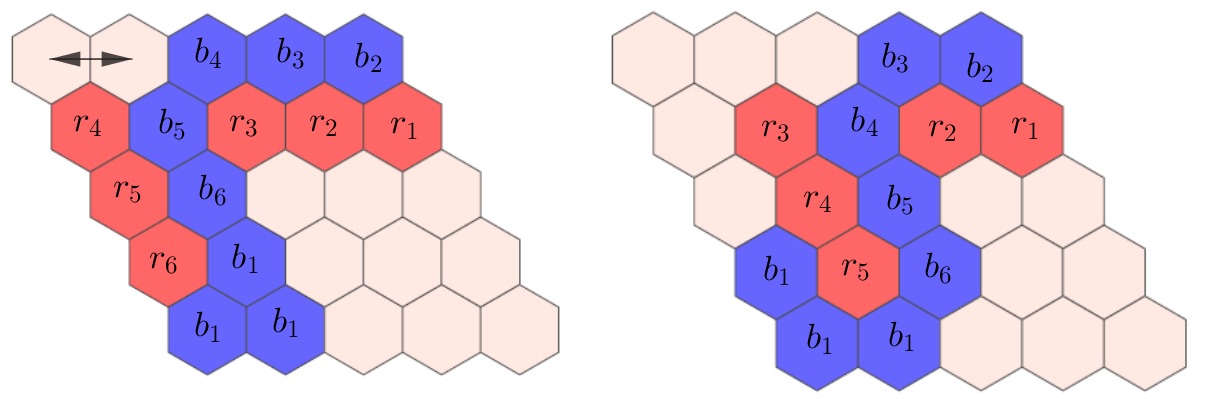}
\caption{The winning path obtained by the Strategy $A_1$ (left) and Strategy $A_2$ (right).}
\label{f-StrategyA1}
\end{center}
\end{figure}

{\bf Strategy $A_2$.} (Assuming $b_1\neq(4,2)$). Suppose that $b_2=(1,5)$, $r_2=(2,4)$ and the cells $(1,1)$ and $(1,2)$ form a $BR$-block. Assume that the game runs only in the rows $1$ and $2$. Then consider the sequence
$$
b_1r_1b_2r_2\rightarrow b_3(1,4)\rightarrow r_3(2,2)\rightarrow b_4(2,3)\rightarrow r_4(3,2)\rightarrow b_5(3,3)\rightarrow r_5(4,2)\rightarrow b_6(4,3)\rightarrow BR-\mbox{block}.
$$

It is clear that if $S$ is either the strategy $A_1$ or $A_2$, then $\lambda(S)\geq7$.

If Blue does not obey the strategies above, then Red cut the path by connecting with the top of the board. Figure \ref{f-StrategyA1} displays the winning path obtained by Strategy $A_1$ and $A_2$ respectively.
\vspace{0.2cm}

{\bf III.} $b_2\in R_3$. If $b_2\not\in X_c^{-}$, then Red can take $r_2(4,4)$ and no winning path is connected with the stones $(2,5), (3,5), (4,5)$,  nor $(5,5)$, obligating to play like in Strategies $A_1$ or $A_2$.\\

Hence we need to consider the cases whenever $b_2\in X_c^{-}$. We have 3 sub-cases.\\
$\bullet$ $b\in X_c^{-}\setminus\{(3,4), (5,3)\}$. In this case, Red chooses $r_2=(4,3)$ and the cells $(3,3)$ and $(3,4)$ form a $BR$-block. 
If $b_3$ is in this $BR$-block, then Red consider the pairs $(2,3)-(2,4)$ and $(1,4)-(1,5)$ as $BR$-blocks and then there are only two possibilities: or Red connected $r_2$ with the top avoiding any winning path finishing in $(1,5), (2,5)$ or the game runs like the above strategies $A_1$ or $A_2$. If $b_3=(5,3)$, Red consider the following strategy:
\vspace{0.2cm}

{\bf Strategy C.} pairs $(5,1)-(5,2)$, $(4,1)-(4,2)$ and $(3,1)-(3,2)$ like $BR$-blocks. In any case, we can observe directelly that $\lambda(S)\geq7$, see Figure \ref{f-StrategyC1}.\\

\begin{figure}[h]
\begin{center}
\includegraphics[width=9.1cm,height=3.1cm,angle=0]{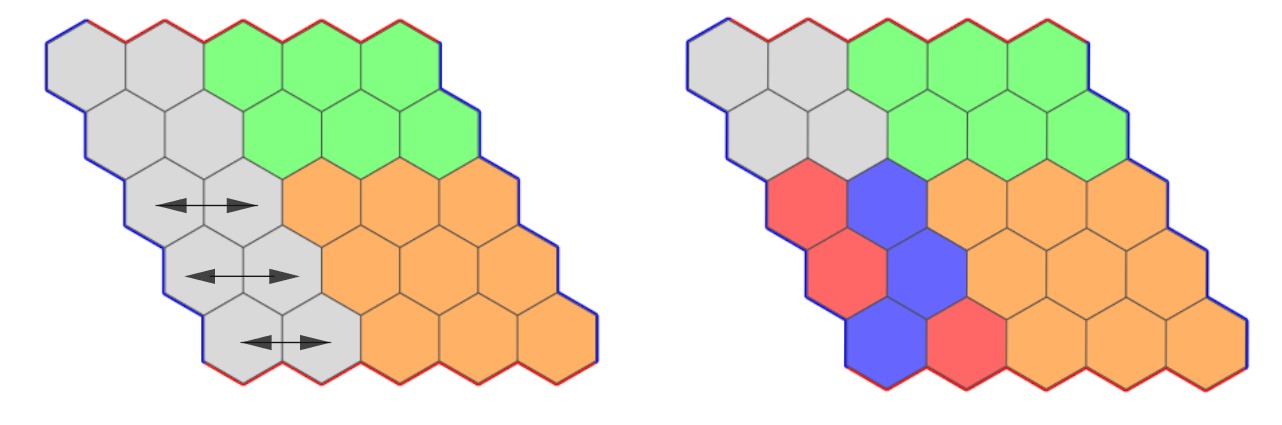}
\caption{Strategy C (left) and a particular case (right).}
\label{f-StrategyC1}
\end{center}
\end{figure}

If $b_2=(5,3)$, then Red take $r_2=(4,4)$. This movement create a $BR$-block formed by the cells $(3,4)$ and $(3,5)$. The Blue must take the cell $b_3=(5,4)$ and follows the Strategy C or to play in the region $R_2$ following the strategy $A_1$ or $A_2$. In any case, a winning path has or $3$ stones in a column or it has two column with $2$ stones.
\vspace{0.2cm}

{\bf IV.} If $b_2=(3,4)$, then Red consider the $BR$-block $(4,3)-(4,4)$ and the next movements are necessary for Blue after the strategical choice of Red
$$
b_1r_1b_2\rightarrow r_2(3,5)\rightarrow b_3(5,4)\rightarrow r_3(4,5)\rightarrow b_4(5,5)\rightarrow r_4(5,3)\rightarrow ...
$$
Then a winning path for Blue contain $3$ cells in the column $4$ or the game runs in the region $R_2$ following strategies $A_1$ or $A_2$. Hence $\lambda(5)\geq7$.\\ 
\vspace{0.2cm}

{\bf Category $\mathfrak{S}_0(5)$.} By the rotational symmetry of the board, any strategy in $\mathfrak{S}_0(5)$ has a symmetry strategy in one of the above categories.

In summary, since for any strategy $S\in \mathfrak{S}(5)$ was prove that $\lambda(S)\geq7$ we can conclude that $\lambda(5)\geq7$ and therefore $\lambda(5)=7$ as desired.\qquad \hspace{9,5cm} $\blacksquare$

\vspace{0.2cm}

\begin{athm} 
(a) There are no winning strategy for Blue starting in the set $$\{(2,1), (3,1), (4,1), (3,2), (5,2), (5,3), (5,4), (4,3)\}.$$
(b) There are no winning strategy in $\mathfrak{S}^+(5)$ if the second choice for Blue is $(2,5)$ or $(3,5)$.
\end{athm}

\begin{athm} $$\delta(5)=7.$$
\end{athm}
{\it Proof.} Follows from the fact that $7=\lambda(5)\leq \delta(5)\leq7.\qquad \hspace{9.5cm} \blacksquare$

\section{The case $5 \times \infty$}\label{s-infty}

We now restrict our attention to an unlimited Hex board. We consider a $n\times\infty$ Hex board $H(5\times\infty)$ where, obviously, the player which was assignated to the finite size always has a winning strategy.   

In the $5\times\infty$ Hex board, we define similar optimal quantifier as in a regular board.

  \begin{defn} The quantity $\lambda(n\times\infty)$ is the smallest integer $k$ such that the winning player has a strategy guaranteeing that every resulting winning path in the $n\times \infty$ Hex board contains at most $k$ stones. 
  \end{defn} 

and also

   \begin{defn} The quantity $\delta(n \times \infty)$ is the minimum, over all winning strategies for the winning player in $H(5\times\infty)$, of the maximum cardinality of a winning set produced against any possible response of the other player. \end{defn} 
   
   Since any strategy $S$ defined in $H(n)$ prescribes a move for Blue after
any admissible sequence of Red replies inside the $n\times n$ window, Blue can simulate $S$ on $H(5\times\infty)$ as follows: whenever Red plays inside the window, Blue answers according to $S$; whenever Red plays outside the window, Blue treats this move as if Red had instead occupied an arbitrary free cell inside the window  and continues to respond according to $S$, while Red's actual stone outside the window plays no role in the connection Blue is building. Since this can only postpone Red's interference with the window, Blue completes the winning path prescribed by $S$ in the window in at most as many moves as on $H(n)$, so one may deduce that
$$
\lambda(n\times\infty)\leq \lambda(n)\qquad\mbox{and}\qquad\delta(n\times\infty)\leq\delta(n),\qquad \mbox{ for any }n\geq1.
$$

Hence, to prove Theorem \ref{t-main2}, it suffices to exhibit a strategy S for Blue that guarantees each winning path to have length at most $5.$
 Our proof achieves this by exploiting a small collection of recurring {\em tactical patterns}. For convenience, we label them Patterns OP, EZ and SB. 
 Each pattern captures a local board configuration together with the corresponding forcing argument. 
 
 We now define the strategy $\mathcal S$ on the $(5\times\infty)$-board and show how these patterns are repeatedly used in its construction. 
 By construction, $\lambda(\mathcal S)=5$. Observe that if $b=(4,j)$ is a cell in the fourth row in a $5\times\infty$-board, then it defines a set $X_b$ consisting of cells $(3,j),(1,j+1),(2,j+1),(3,j+1),(1,j+2)$. This set appears because the points $b$ and $(2,2)$ create a winning virtual connection at $b$. This means that if $r_1\not\in X_b$ then Blue player wins the game defining $b_2=(2,j+1)$.    

We define the first movement of blue as $b_1=(4,1)$ (or any $(4,j)$ because its choice will does not depend on the column coordinate).
For this cell we look at $X_{b_1}$. If $r_1\not\in X_{b_1}$ the we define $b_2$ in our strategy $S$ by $b_2=(2,2)$ and Blue wins the game. Obviously the winning path has length $5$.\\

It remains to define the strategy $\mathcal{S}$ when $r_1$ belongs to $X_{b_1}$. We need to split the set $X_{b_1}$ into two classes and to analyze each class separately. We define the subsets of $X_{b_1}$, 

$\mathcal{O}_{b_1}^+=\{(3,1),(1,2),(2,2)\}$ and we say that $r_1$ is {\bf open to the right} (or {\bf close to the left}). \\

$\mathcal{O}_{b_1}^-=\{(2,2),(3,2),(1,3)\}$ and  $r_1$ is said to be {\bf open to the left} (or {\bf close to the right}).\\

The point $(2,2)$ is open to the right and left simultaneously. See Figure \ref{f-Zzone}.\\

\begin{figure}[h]
\begin{center}
\includegraphics[width=5.0cm,height=4.0cm,angle=0]{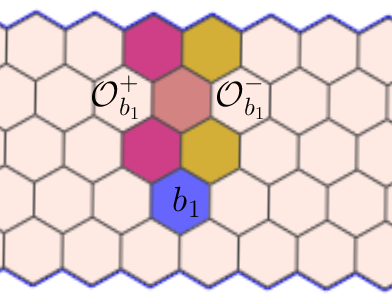}
\caption{The set $X_{b_1}^+$ divided into the subsets $\mathcal{O}_{b_1}^+$ and $\mathcal{O}_{b_1}^-$.}\label{f-Zzone}
\end{center}
\end{figure}

Note that if $b_2$ is of the form $(4,j)$, $j\geq3$, then $X_{b_1}^+\cap X_{b_2}^+=\emptyset$. The same is true for $j\leq -2$. Then we shall take $b_2$ of the form $(4,2j+1)$ or $(4,-2j)$, $j\geq4$, and analyze the choices of $r_1$ and $r_2$. For a description of the strategy $\mathcal{S}$ some tecnical  moves are necessary to expose. Most of them are virtual connections.


$\bullet$ {\bf Zig-Zag method.} Suppose that $b_2=(2,-2j)$, $j\geq1$ and $r_1,r_2\not\in\{(k,l):-2j-1\leq l\leq 2\}$. Then the next blue movements are showed in the following sequence. Red choices are necessary to avoid any early winning path of length $5$.
\begin{eqnarray*}
b_3(4,-2j)\rightarrow r_3(3,-2j)&\rightarrow& b_4(2,-2j+2)\rightarrow r_4(3,-2j+1)\rightarrow\\
&&\rightarrow b_5(4,-2j+2)\rightarrow r_5(3,-2j+2)\rightarrow b_6(2,-2j+4)...
\end{eqnarray*}
Continuing with these Zig-Zag connections, the point $b_{2(j+1)}=(2,1)$ yields a winning virtual connection at it through $(3,1)$ and $(3,-1)$. Clearly, any permissible winning path has length equal to $5$, see Figure \ref{f-Zigzag}.

An equivalent Zig-Zag method can be defined from the left to the right. Note also that $b_2$ can be of the form $(4,-2j)$.\\

\begin{figure}[h]
\begin{center}
\includegraphics[width=6.5cm,height=3.0cm,angle=0]{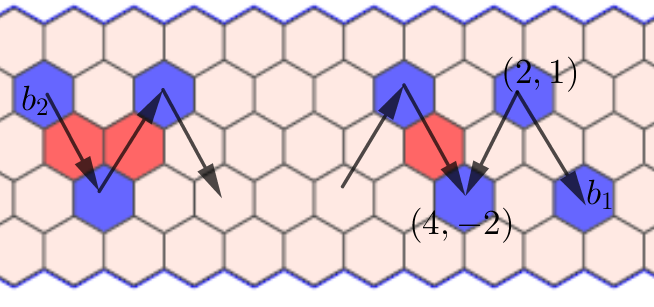}
\caption{The Zig-Zag method.}\label{f-Zigzag}
\end{center}
\end{figure}

The next method connects the two first movements $b_1$ and $b_2$ of Blue when $r_1$ and $r_2$ are in sets of different polarity meaning that the left Red response is open to the right and the right choice is open to the left. 
\medskip

\noindent  [{\bf Pattern OP (Opposite Polarity)}] Assume that $r_1$ is in $\mathcal{O}_{b_1}^+$, $b_2(4,i+2j)$ and $r_2\in\mathcal{O}^-_{b_2}$. Then Blue can creates the same above zig-zag performance and then $b_{2,i}$ forms two independent virtual connection guaranteeing a victory for Blue. Again the two virtual winning paths have length $5$.

In analogy, the above can be applied to the case $r_1\in \mathcal{O}_{b_1}^-$, $b_2(4,-2j)$, and $r_2\in\mathcal{O}^+_{b_2}$.

Now, to define $b_2$ we need to consider the cases when $r_1$ is in either $\mathcal{O}_{b_1}^+$ or $\mathcal{O}_{b_1}^-$.

{\bf Case I.} $r_1\in \mathcal{O}_{b_1}^+$. Then define blue $b_2=(4,2j+1)$, $j\geq4$. We divide this case into two subcases depending on the position of $r_2$ in $X_{b_2}^+$:\\

$\bullet$ $r_2$ is open to the left. The strategy $\mathcal{S}$ is defined according to Pattern OP.\\

\noindent [{\bf Pattern EZ (Extended Zig-Zag)}]

$\bullet$ $r_2$ is open to the right. 
If  $r_2=(2,2j+1)$ then it is also open to the left, and it was already considered above.

 Assume now that  $r_2=(3,2j+1)$. Since $r_1\in\mathcal{O}_{b_1}^+$, it is possible to apply the above zig-zag strategy up to $b_{2j+1}=(2,2j-1)$ and then $b_{2j+2}=(3,2j)$ implies $r_{2j+2}=(1,2j+1)$. \\
Now Blue can take $b_{2j+3}=(3,2j+2)$ and obligate the choice of Red to be $r_{2j+3}(1,2j+3)$.\\
But the cell $b_{2j+4}=(1,2j+2)$ creates two independent virtual connections. It is clear that the two possible winning path have length $5$. Note that this method works even for $j=1$. See Figure \ref{f-patternEZ}.

\begin{figure}[h]
\begin{center}
\includegraphics[width=11.5cm,height=3.0cm,angle=0]{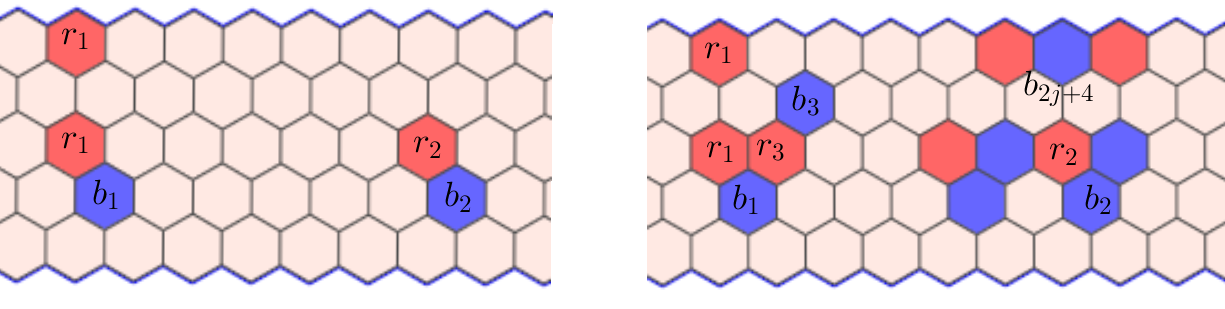}
\caption{Pattern EZ} \label{f-patternEZ}
\end{center}
\end{figure}

Now suppose that $r_2=(1,2j+2)$. Take $b_3=(4,3)$ and by the above configuration, $r_3=(1,4)$. By the same argument, we  define $b_4=(4,5)$ and then $r_4=(1,6)$. Also define $b_5=(4,7)$ and then $r_5=(1,8)$. But this situation falls into the following pattern.  
\medskip

\noindent [{\bf Pattern SB (Three-Step Barrier)}] Suppose that the first three movements of each player have been placed in the following sequence
$$
b1r_1(1,2)\rightarrow b_2(4,3)\rightarrow r_2(1,4) \rightarrow b_3(4,5)\rightarrow r_3(1,6).
$$ 
Blue can take $b_4(1,3)$ and then the next sequence is mandatory for Red for avoiding any winner path:\\
$$
b_4\rightarrow r_4(3,2)\rightarrow b_5(3,3)\rightarrow r_5(2,3)\rightarrow
$$
But now the cell $b_6(1,5)$ win the game, because Blue creates two independent virtual connections, one in the $b_2$ through $(2,4)$ and the other with the two independent virtual connections in $b_2$ and $b_3$ through $(2,6)$. Note that both possible winning path has length $5$. See Figure \ref{f-patternSB}.

\begin{figure}[h]
\begin{center}
\includegraphics[width=10.0cm,height=3.0cm,angle=0]{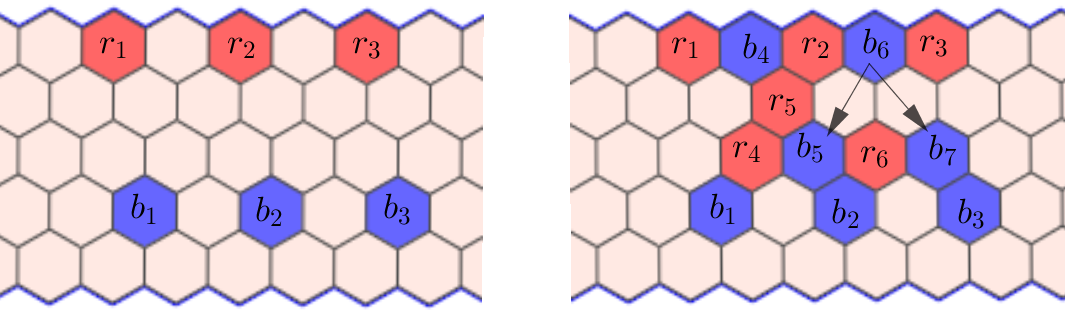}
\caption{Pattern SB}\label{f-patternSB}
\end{center}
\end{figure}

The definition of $\mathcal{S}$ in the first case is now complete.\hspace{6.5cm} \qquad$\blacksquare$
\medskip

Now we shall analyze the second case. In this conditions, the structure of the strategy $\mathcal{S}$ is reduced to the previous case. 

 {\bf Case II. $r_1\in\mathcal{O}^-_{b_1}$.} Then the second blue movement is in to the left of $b_1$ but in the second line. It is defined $b_2=(2,-2)$. We again break down the case into two subcases depending on the $r_2$ open possibilities. 

$\bullet$ If $r_2\in \widetilde{\mathcal{O}}_{b_2}^-:=\{(3,-3), (5,-4)\}$, then the following sequence is necessary
$$
b_1r_1b_2r_2\rightarrow b_3(4,-2)\rightarrow r_3(3,-2)\rightarrow b_4(2,1). 
$$
Then $b_4$ forms a winning virtual connection at $b_4$ passing through $b_1$ and $b_3$. It is clear that any winning path formed here has length $5$.\\

$\bullet$ If $r_2\in\widetilde{\mathcal{O}}_{b_2}^+:=\{(3,-2), (5,-3)\}$, then we apply the first part assuming $b_2$ as the first Blue movement. As notated there, any winning path has length $5$, see Figure \ref{f-caso2}.\\

In summary, Cases I and II together define a strategy $\mathcal{S}$ in the Hex board $H(5\times\infty)$ such that any winning path has length $5$. This implies that $\lambda(5\times\infty)=5,$ concluding the proof of Theorem~\ref{t-main2}. 

\begin{figure}[h]
\begin{center}
\includegraphics[width=4.0cm,height=3.0cm,angle=0]{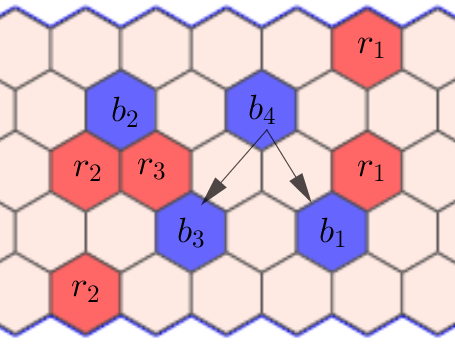}
\caption{Case II in the proof of Theorem \ref{t-main2}.}
\label{f-caso2}
\end{center}
\end{figure}
 \medskip 
  
In our especial case $n=5$, we showed that $\lambda(5)=7$ and $\lambda(5\times\infty)=5$. This says that $\lambda(n\times\infty)$ can be strictly less than $\lambda(n)$. At the same time, we obtained that $\delta(5)=7$ in the regular $5\times 5$ board. 
In the $H(5\times\infty)$, note that exactly as in the finite case, every winning path on the
$5\times\infty$ board must meet each of the $5$ rows of the board at least
once, so $\lambda(5\times\infty)\ge 5$. Since $\delta(5\times\infty) \geq \lambda(5\times\infty)$ we conclude
\[
\delta(5\times\infty)\ge 5.
\]
But $\delta(5\times\infty)$ can be not equal to $5$ because, in contrary, every stone played by Blue in $H(5\times\infty)$ would belong to a same board of size $5$, and this would happen only if the game never leaves a window of $5$ consecutive columns. Then this restricted game induces a winning strategy in $H(5)$ of length $5$ contradicting that $\lambda(5)=7$. This shows, in particular, that $5=\lambda(5\times\infty) < \delta(5\times\infty)\leq\delta(5)=\lambda(5)=7$. We do not know the exact value of $\delta(5\times\infty)$.

\section{Concluding remarks and open problems}\label{s-finalremarks}
This paper settles two conjectures of Campbell \cite{Ca04}
concerning the complexity parameters $\lambda(n)$ and $\delta(n)$
associated with optimal play in Hex.

We have determined the exact value $\delta(5)=7$ and $\lambda(5)=7$.
Together these results confirm that
\[
\lambda(5) = \delta(5) = 7 = 2 \cdot 5-3,
\]
extending the pattern $\lambda(n)=\delta(n)=2n-3$ previously known
for $n\in\{3,4\}$.
Several natural questions remain open.

\subsection*{The bound $\lambda(n)=2n-3$ for larger $n$}

The values computed so far are consistent with the conjecture
\[
\lambda(n) = \delta(n) = 2n-3 \quad \text{for all } n \ge 3.
\]
The upper bound $\lambda(n)\le 2n-3$ for $n\in\{3,4,5\}$
is established in Section~\ref{s-proof-teoA} via the centre strategy.
It would be interesting to verify whether the same strategy
achieves the bound for $n=6$ and $n=7$.

\subsection*{Exact values for $n=6$}

The most immediate open problem is to determine the exact values of
$\lambda(6)$ and $\delta(6)$.
The conjecture $\lambda(6)=\delta(6)=9$ is consistent with all
known data.
Establishing the upper bound $\lambda(6)\le 9$ via the centre
strategy requires verifying that the east and west blocking numbers
on the $6\times6$ board both exceed~$5$,
which would follow from an enumeration of routes of length $\le 4$
from the centre cell $(3,3)$.

\subsection*{Exact value of $\delta(5\times\infty)$}

Theorem $B$ gives the exact value of $\lambda(5\times\infty)$.  What is the exact value of $\delta(5\times \infty)$? Note that was already observed the inequalities $6\leq\delta(5\times\infty)\leq 7$.   

\subsection*{Equality of the parameters}

For all $n\le 5$, the two parameters coincide: $\lambda(n)=\delta(n)$.
It is an open problem whether this equality holds for all $n$,
or whether the two parameters eventually diverge.
A proof of $\lambda(n)=\delta(n)$ for all $n$ would imply that
the shortest winning path and the minimum winning set always have
the same size, meaning that every stone in an optimal winning path
is essential.  
This question, however, is specific to finite boards: on the infinite strip
the two parameters already diverge at $n=5$, since we showed that
$\lambda(5\times\infty) < \delta(5\times\infty)$.

\subsection*{Asymptotic behaviour}

If the conjecture $\lambda(n)=2n-3$ holds for all $n\ge 3$,
then $\lambda(n)/n\to 2$ as $n\to\infty$.
Even without this conjecture, it would be interesting to determine
the true asymptotic growth rate of $\lambda(n)$ and $\delta(n)$,
and in particular whether $\lambda(n)/n$ converges. The same may be asked for unlimited boards.



\begin{thebibliography}{10}
\bibitem{Ca04}
Garikai Campbell 
\newblock On optimal play in the game of Hex.
\newblock {\em Eletronic  Journal of Combinatorial  Number Theory,} 4, 2004.

\bibitem{ET76}
Shimon Even and Robert E. Tarjan.
\newblock A combinatorial problem which is complete in polynomial space.
\newblock {\em Journal of the ACM}, 23(4): 710--719, 197

\bibitem{Ga79}
\newblock David Gale. 
\newblock The game of Hex and the Brouwer fixed point theorem. 
\newblock {\em American Mathematical Monthly,} 86(10): 818 – 827, 1979

\bibitem{HAH09}
\newblock Henderson, P., Arneson, B., Hayward, R.B.:
\newblock Solving 8x8 Hex.
\newblock {\em Boutilier, C. (ed.)
IJCAI,} 505--510, 2009.


\bibitem{HR06}
\newblock Ryan B. Hayward, Jack van Rijswijck.
\newblock Hex and combinatorics.
\newblock {\em Discrete Mathematics,} 306: 2515--2528, 2006.


\bibitem{HBJKPR05}
Ryan Hayward, Yngvi Bj\"ornsson, Michael Johanson, Morgan Kan,
Nathan Po, and Jack van Rijswijck.
\newblock Solving $7\times7$ Hex with domination, fill-in, and virtual connections.
\newblock {\em Theoretical Computer Science}, 349:125--139, 2005.

\bibitem{Mi06}
\newblock Mishita, K., Sakurai, H., Noshita, K.
\newblock New proof techniques and their applications to winning strategies in Hex.
\newblock {\em Proceedings of 11th game programing workshop in Japan}, 136--142, 2006.

\bibitem{No04}
\newblock Noshita, K.
\newblock Union-connections and a simple readable winning way in $7\times 7$ Hex.
\newblock {\em Proceedings of 9th game programming workshop in Japan}, 72--79, 2004.

\bibitem{No05}
\newblock Noshita, K.
\newblock Union-connections and straightforward winning strategies in Hex.
\newblock {\em ICGA Journal} 28(1), 3--12, 2005.

\bibitem{Rei81}
\newblock Stefan Reisch.
\newblock Hex ist PSPACE-vollst\"andig.
\newblock {\em Acta Informatica}, 15: 167--191, 1981.

\bibitem{St07}
\newblock Walter Stromquis.
\newblock Winning paths in $N$-By-Infinity Hex.
\newblock {\em Integers: Eletronic Journal of Combinatorial Number Theory}, 7, 2007.

\bibitem{Ya01}
\newblock Yang, J., Liao, S., Pawlak, M.
\newblock A decomposition method for finding solution in
game Hex $7\times 7$.
\newblock {\em Ning, C.T. (ed.) ADCOG, City University of Hong Kong,} 96--111, 2001.
\end{thebibliography}
\end{document}